\documentclass[11pt]{article}

\usepackage{mystyle}

\title{Improving the Caro-Wei bound and applications to Tur\'{a}n stability}

\author{Tom Kelly%
  \thanks{
    School of Mathematics, Georgia Institute of Technology. 
    Email: \protect\href{mailto:tom.kelly@gatech.edu}{\protect\nolinkurl{tom.kelly@gatech.edu}}.
    Research
supported by the National Science Foundation under Grant No. DMS-224707
  }
  \and
  Luke Postle%
  \thanks{
    Department of Combinatorics and Optimization, University of Waterloo, Canada.
    Email: \protect\href{mailto:lpostle@uwaterloo.ca}{\protect\nolinkurl{lpostle@uwaterloo.ca}}.
    Partially supported by NSERC under Discovery Grant No.\ 2019-04304, the Ontario Early Researcher Awards program and the Canada Research Chairs program.
  }}

\begin{document}
\maketitle

\begin{abstract}
  We prove that if $G$ is a graph and $f(v) \leq 1/(d(v) + 1/2)$ for each $v\in V(G)$, then either $G$ has an independent set of size at least $\sum_{v\in V(G)}f(v)$ or $G$ contains a clique $K$ such that $\sum_{v\in K}f(v) > 1$.  This result implies that for any $\sigma \leq 1/2$, if $G$ is a graph and every clique $K\subseteq V(G)$ has at most $(1 - \sigma)(|K| - \sigma)$ simplicial vertices, then $\alpha(G) \geq \sum_{v\in V(G)} 1 / (d(v) + 1 - \sigma)$.  Letting $\sigma = 0$ implies the famous Caro-Wei Theorem, and letting $\sigma = 1/2$ implies that if fewer than half of the vertices in each clique of $G$ are simplicial, then $\alpha(G) \geq \sum_{v\in V(G)}1/(d(v) + 1/2)$, which is tight for the 5-cycle.  When applied to the complement of a graph, this result implies the following new Tur\' an stability result.  If $G$ is a $K_{r + 1}$-free graph with more than $(1 - 1/r)n^2/2 - n/4$ edges, then $G$ contains an independent set $I$ such that at least half of the vertices in $I$ are complete to $G - I$.  Applying this stability result iteratively provides a new proof of the stability version of Tur\' an's Theorem in which $K_{r + 1}$-free graphs with close to the extremal number of edges are $r$-partite. 
\end{abstract}

\section{Introduction}

For $n \geq r \geq 2$, the \textit{Tur\' an graph}, which we denote $T_r(n)$, is the complete $r$-partite graph on $n$ vertices whose part sizes differ by at most one.  Tur\' an's Theorem~\cite{T41}, the cornerstone of extremal graph theory, states that if $G$ is an $n$-vertex, $K_{r + 1}$-free graph, then $e(G) \leq e(T_r(n))$, with equality holding only when $G \cong T_r(n)$.  Since $e(T_r(n)) \leq (1 - 1/r)n^2/2$, Turan's Theorem is sometimes presented in the following slightly weaker form.  If $G$ is a $K_{r+1}$-free graph on $n$ vertices, then $G$ has at most $(1 - {1}/{r}){n^2}/{2}$ edges.  We refer to this latter statement as the \textit{concise Tur\'an's Theorem} (as in~\cite{N11}).

A stability version of Tur\' an's Theorem implies that if $G$ is an $n$-vertex $K_{r+1}$-free graph with close to $(1 - 1/r)n^2 / 2$ edges, then $G$ resembles the Tur\' an graph $T_r(n)$ in some sense.  For example, the classical Erd\H os--Simonovits~\cite{E66, E67, S66} stability theorem states that if $G$ has at least $e(T_r(n)) - o(n^2)$ edges, then the \textit{edit distance} from $G$ to $T_r(n)$ is $o(n^2)$.  That is, $G$ can be obtained from $T_r(n)$ by adding and deleting at most $o(n^2)$ edges.  Brouwer~\cite{Br81} proved that for $n \geq 2r + 1$, if $G$ has at least $e(T_r(n)) - \lfloor n/r\rfloor + 2$ edges, then $G$ is $r$-partite.  In this paper, we prove a new stability version of Tur\' an's Theorem that is qualitatively between these two.  Our result (Theorem~\ref{strong local turan stability}) implies that if $G$ has close to the extremal number of edges, then $G$ has an independent set $I$ which resembles one of the parts of the Tur\' an graph in that it contains many vertices that are complete to $G - I$.  We call this phenomenon ``Local Tur\' an stability.''

When applied to the complement of a graph, the concise Tur\' an's Theorem states that an $n$-vertex graph with less than $n^2/(2r) - n/2$ edges has an independent set of size $r + 1$ .  Equivalently, every graph with $n$ vertices and average degree $d$ has an independent set of size at least $n/(d + 1)$.  Independently, Caro~\cite{C79} and Wei~\cite{W81} famously generalized the concise Tur\'{a}n's Theorem by proving that every graph $G$ has an independent set of size at least $\sum_{v\in V(G)}1/(d(v) + 1)$, where $d(v)$ is the degree of the vertex $v$ in $G$.  This bound is stronger than Tur\'{a}n's because of the convexity of the function $x\mapsto 1/(x + 1)$.  Our main result in this paper, Theorem~\ref{main thm} (which implies Theorem~\ref{strong local turan stability}), is a best-possible improvement of the Caro-Wei bound.

\subsection{Local Tur\' an stability}

The following result is our new stability version of Tur\' an's Theorem.  We say that a vertex in a graph is \textit{complete} to a subset of vertices if the vertex is adjacent to every vertex in that set.

\begin{theorem}[Local Tur\'{a}n Stability]\label{strong local turan stability}
  Let $\sigma \leq 1/2$.  If $G$ is a $K_{r + 1}$-free graph on $n$ vertices with more than
  \begin{equation*}
    \left(1 - \frac{1}{r}\right)\frac{n^2}{2} - \frac{\sigma n}{2}
  \end{equation*}
  edges, then $G$ contains an independent set $I$ such that more than $(1 - \sigma)(|I| - \sigma)$ vertices in $I$ are complete to $G - I$.
\end{theorem}

Theorem~\ref{strong local turan stability} for $\sigma = 0$ implies the concise Tur\' an's Theorem; indeed, no graph contains an independent set $I$ with more than $|I|$ vertices in $I$, so by Theorem~\ref{strong local turan stability}, no $K_{r + 1}$-free $n$-vertex graph has more than $(1 - 1/r)n^2 / 2$ edges.  For other values of $\sigma$, the result suggests that a $K_{r+1}$-free graph with close to the extremal number of edges resembles the Tur\' an graph in that the independent set $I$ behaves like one of the parts of the Tur\' an graph.

Theorem~\ref{strong local turan stability} is essentially best possible for $r = 2$ in the sense that if $n$ is even and $\sigma n/2$ is an integer, then the complete bipartite graph on $n$ vertices with equal-sized parts and a matching of size $\sigma n/2$ removed has precisely $(1 - 1/r)n^2/2 - \sigma n/2$ edges and every independent set $I$ has at most $(1 - \sigma)|I|$ vertices that are complete to $G - I$ (to see this last fact, observe that if $I$ is an independent set with at least one vertex complete to $G - I$, then $I$ is one of the two parts of the bipartition).  Moreover, the statement of Theorem~\ref{strong local turan stability} is not true with $\sigma > 1/2$ because if $G$ is the 5-cycle, then $(1 - 1/r)n^2/2 - \sigma n/2 < 5$, but $G$ does not contain an independent set with any vertex complete to the rest of the graph.

In an application of Theorem~\ref{strong local turan stability}, the subgraph $G - I$ is $K_r$-free, so if we apply Theorem~\ref{strong local turan stability} iteratively, then we can show that $G$ globally resembles the Tur\' an graph in that it is actually $r$-partite, as follows.

\begin{corollary}\label{strong turan stability}
  If $G$ is a $K_{r + 1}$-free graph on $n$ vertices with at least
  \begin{equation*}
    \left(1 - \frac{1}{r}\right)\frac{n^2}{2} - \frac{n}{2r} + 1
  \end{equation*}
  edges, then $G$ is $r$-colorable.
\end{corollary}

For completeness, we include a proof of Corollary~\ref{strong turan stability} in Appendix~\ref{strong turan section}.  The only complicated aspect of the proof is that it is necessary to solve a simple optimization problem to determine how many times Theorem~\ref{strong local turan stability} is iterated and how many vertices remain at the end of the process.

Considerable attention has been devoted to various aspects of Tur\' an stability for $K_{r+1}$-free graphs as the number of edges approaches the extremal value.  Theorem~\ref{strong local turan stability} and Corollary~\ref{strong turan stability} fit into this paradigm.  Table~\ref{phase-transition-table} illustrates this ``phase transition''\footnote{The results of~\cite{PSS18, TU15} apply to \textit{maximal} $K_{r+1}$-free graphs.}.
See also~\cite{F15, NR04}.

\bgroup
\def\arraystretch{1.25}%
\begin{table}
  \centering
  \begin{tabular}[t]{| c | c | c |}
    \hline
    $t$ & Property & Reference\\
    \hline
    $0$ & $G\cong T_r(n)$ & Tur\' an's Theorem~\cite{T41}\\
    \hline
    $\lfloor n/r\rfloor + 2$ & $\chi(G) \leq r$ & Brouwer~\cite{Br81}\\
    \hline
    $n/4$ & Local Tur\' an stability & Theorem~\ref{strong local turan stability} \\
    \hline
    $O(n \log n)$ & $\exists$ a pair of ``twin vertices'' & \cite{TU15}\\
    \hline
    $o(n^{(r+1)/r})$ & $\exists$ almost spanning complete $r$-partite subgraph & \cite{PSS18} \\
    \hline
    $o(n^2)$ & edit distance $o(n^2)$ to $T_r(n)$ & Erd\H os--Simonovits~\cite{E67, S66}\\
    \hline                   
  \end{tabular}
  \caption{``Phase transition'' of Tur\'an stability: $G$ is $K_{r + 1}$-free with $e(G) \geq (1 - \frac{1}{r})\frac{n^2}{2} - t$.}
  \label{phase-transition-table}
\end{table}
\egroup

As mentioned, Corollary~\ref{strong turan stability} was first proved by Brouwer~\cite{Br81} -- it has actually been reproved by various sets of authors (see~\cite{TU15}).  For $n \geq 2r + 1$ the result actually holds with the $-n/(2r) + 1$ term replaced with $- \lfloor n/r\rfloor + 2$.  We could possibly obtain this stronger form of Corollary~\ref{strong turan stability} via the same proof method with an extension of Theorem~\ref{strong local turan stability} with $\sigma \leq 1$ if we include an additive ``error term.''  More precisely, we conjecture that the conclusion of Theorem~\ref{strong local turan stability} still holds with $\sigma \leq 1$ if we replace ``$(1 - 1/r)n^2/2 - \sigma n/2$ edges'' with ``$(1 - 1/r)n^2/2 - \sigma n/2 + 5/4$ edges.''  Moreover, we believe it is possible to characterize the extremal examples which would in turn provide a new proof of a stronger form of Brouwer's~\cite{Br81} result also characterizing the extremal examples.  We believe the extremal examples for Theorem~\ref{strong local turan stability} are the same as those for Brouwer's~\cite{Br81} result for $r = 2$, which are particular blowups of 5-cycles -- see~\cite{TU15} for a complete description.  The complement of the graph on the left in Figure~\ref{fig:blowups} provides an example.  For $r \geq 3$, the extremal examples for Brouwer's result are obtained from an extremal example for $r - 1$ by adding an independent set complete to the rest of the graph, so these graphs satisfy Theorem~\ref{strong local turan stability} with $\sigma = 1$.  

\subsection{Improving the Caro-Wei bound}

The \textit{independence number} of a graph $G$, denoted $\alpha(G)$, is the size of a largest independent set in $G$.  As witnessed by Tur\' an's Theorem~\cite{T41} and the Caro-Wei Theorem~\cite{C79, W81}, some of the most fundamental results in graph theory include bounds on the independence number.  The bound on the independence number provided by the Caro-Wei Theorem is tight for disjoint unions of complete graphs, and the complement of the Tur\'{a}n graph is a disjoint union of complete graphs of order differing by at most 1.  Several improvements have been made to these bounds under certain additional assumptions, such as connectivity~\cite{BJ08, GV10, HR11, HS01, HS06} or triangle-freeness~\cite{AKS80, G83, Sh83, Sh91}, among other things~\cite{BGHR12, BR15, CRRS16, HM19, S94}.  Brooks' Theorem~\cite{B41} bounding the chromatic number implies that if $G$ is an $n$-vertex connected graph of maximum degree $\Delta$, then $\alpha(G) \geq n/\Delta$, unless $G$ is a complete graph or an odd cycle.  

In this paper, we examine the extent to which Brooks' bound holds for the average degree as in the concise Tur\' an's Theorem instead of the maximum degree, or for the degree-sequence as in the Caro-Wei Theorem.  In particular, for $\sigma \in [0, 1]$, when does a connected $n$-vertex graph $G$ of average degree $d$ satisfy $\alpha(G) \geq n/(d + 1 - \sigma)$ or even $\alpha(G) \geq \sum_{v\in V(G)}1/(d(v) + 1 - \sigma)$?  For any $\sigma > 0$, it is not sufficient as in Brooks' Theorem to simply require that $G$ is not complete  -- a disjoint union of complete graphs with the minimum number of edges added to ensure connectivity has independence number at most $n/(d + 1 - O(1/d))$.  However, the cliques in this graph have many simplicial vertices (a vertex is \textit{simplicial} if its neighborhood is a clique).  We prove that for $\sigma \leq 1/2$, such cliques are essentially the only obstruction, as follows.

\begin{theorem}\label{technical independence brooks}
  Let $\sigma \leq 1/2$.  If $G$ is a graph and each clique $K\subseteq V(G)$ has at most $(1 - \sigma)(|K| - \sigma)$ simplicial vertices, then
  \begin{equation*}
    \alpha(G) \geq \sum_{v\in V(G)}\frac{1}{d(v) + 1 - \sigma}.
  \end{equation*}
  In particular, if each clique $K\subseteq V(G)$ has fewer than $|K|/2$ simplicial vertices, then $\alpha(G) \geq \sum_{v\in V(G)}1/(d(v) + 1/2)$.  Consequently, if $G$ also has $n$ vertices and average degree $d$, then
  \begin{equation*}
    \alpha(G) \geq \frac{n}{d + 1/2}.
  \end{equation*}
\end{theorem}

Notice that the Caro-Wei Theorem follows from Theorem~\ref{technical independence brooks} with $\sigma = 0$.  For graphs without even a single simplicial vertex, no bound on $\alpha(G)$ of the form $n / (d + 1 - \varepsilon)$ was previously known for any $\varepsilon > 0$.  Theorem~\ref{technical independence brooks} implies this result for $\varepsilon = 1/2$, and it is tight for the 5-cycle.  Theorem~\ref{technical independence brooks} also immediately implies Theorem~\ref{strong local turan stability}.

Our main result is actually the following generalization of Theorem~\ref{technical independence brooks}\footnote{This result follows from a previous result of both authors~\cite[Theorem~1.3]{KP18} on fractional coloring, but that paper is currently not under review.  In this paper, we provide a much shorter proof of our main result.}.

\begin{theorem}\label{main thm}
  If $G$ is a graph and $f : V(G) \rightarrow \mathbb R$ such that $f(v) \leq 1/(d(v) + 1/2)$ for each $v\in V(G)$ and $\sum_{v\in K} f(v) \leq 1$ for each clique $K\subseteq V(G)$, then $\alpha(G) \geq \sum_{v\in V(G)} f(v)$.
\end{theorem}

The condition in Theorem~\ref{main thm} that every clique $K\subseteq V(G)$ satisfies $\sum_{v\in K}f(v) \leq 1$ is very natural.  If $G$ is a complete graph, then it is a necessary condition.  The stronger statement of Theorem~\ref{main thm} is also more conducive to our proof by induction, since subgraphs of $G$ also satisfy the hypotheses of the theorem.  In a follow-up paper~\cite{KP19-brooks}, we generalize Theorem~\ref{main thm} to \textit{fractional coloring}.

Now we show how Theorem~\ref{main thm} implies Theorem~\ref{technical independence brooks}.
\begin{proof}[Proof of Theorem~\ref{technical independence brooks} assuming Theorem~\ref{main thm}]
  Let $f : V(G) \rightarrow \mathbb R$ be the function where $f(v) = 1/(d(v) + 1 - \sigma)$ for each $v\in V(G)$.  Since $\sigma \leq 1/2$, by Theorem~\ref{main thm}, it suffices to show that if $K$ is a clique, then $\sum_{v\in K}f(v) \leq 1$.  Each $v\in K$ satisfies $f(v) \leq 1/(|K| - \sigma)$ and moreover $f(v) \leq 1/(|K| + 1 - \sigma)$ if $v$ is not simplicial.  Hence, since $K$ has at most $(1 - \sigma)(|K| - \sigma)$ simplicial vertices,
  \begin{equation*}
    \sum_{v\in K}f(v) \leq \frac{(1 - \sigma)(|K| - \sigma)}{|K| - \sigma} + \frac{|K| - (1 - \sigma)(|K| - \sigma)}{|K| + 1 - \sigma} = 1,
  \end{equation*}
  as required.  Therefore by Theorem~\ref{main thm}, $\alpha(G) \geq \sum_{v\in V(G)}f(v) = \sum_{v\in V(G)}1/(d(v) + 1 - \sigma)$, as desired.
  
  Also note that since $|K|$ is an integer, if $K$ has fewer than $|K|/2$ simplicial vertices, then $K$ has at most $(|K| - 1)/2$ simplicial vertices.  If $\sigma = 1/2$, then $(|K| - 1)/2 = (1 - \sigma)(|K| - \sigma)$, so if each clique $K \subseteq V(G)$ has fewer than $|K| / 2$ simplicial vertices, then $\alpha(G) \geq \sum_{v \in V(G)}1 / (d(v) + 1/2)$, as desired.
\end{proof}

As with Theorem~\ref{strong local turan stability}, it would be interesting to extend Theorems~\ref{technical independence brooks} and~\ref{main thm} for $\sigma \leq 1$ while taking into account additional obstructions.  As Brooks' Theorem suggests, it is necessary to consider odd cycles.  However, blowups of 5-cycles (depicted on the left in Figure~\ref{fig:blowups}) are in some sense ``worse'' obstructions, as follows.  Let $G$ be a graph with $n$ vertices and average degree $d$.  If $G$ is a cycle of length $2k + 1$, then $n/(d + 1 - \sigma) \leq \alpha(G) = k$ only if $\sigma \leq 1 - 1/k = 1 - 2/(n - 1)$, and if $G$ is the graph on the left in Figure~\ref{fig:blowups}, which is obtained by replacing two non-adjacent vertices of $C_5$ with cliques of size $t$, then $n/(d + 1 - \sigma) \leq \alpha(G) = 2$ only if $\sigma \leq 1 - 5/(4t + 6) = 1 - 5/(2n)$.  Moreover, the graphs in Figure~\ref{fig:blowups} are not regular, so if $G$ is one of these graphs, then $\sum_{v\in V(G)}1/(d(v) + 1 - \sigma) < n / (d + 1 - \sigma)$.  If $G$ is the graph on the left, then $\sum_{v\in V(G)}1/(d(v) + 1 - \sigma) \leq 2$ only if $\sigma < 3/4$, for large $t$, and if $G$ is the graph on the right, which is obtained by replacing three pairwise non-adjacent vertices of $C_7$ with cliques of size $t$, then $\sum_{v\in V(G)}1/(d(v) + 1 - \sigma) \leq 3$ only if $\sigma < 1$ for large $t$.  We conjecture that these are essentially the only obstructions, as follows.

\begin{conjecture}
  If $G$ is a graph and $f : V(G) \rightarrow \mathbb R$ such that $f(v) \leq 1/d(v)$ for each $v\in V(G)$, then $\alpha(G) \geq \sum_{v\in V(G)} f(v)$ unless $G$ contains either
  \begin{itemize}
  \item a clique $K \subseteq V(G)$ such that $\sum_{v\in K} f(v) > 1$, or
  \item a subgraph $H$ isomorphic to a blowup of a cycle of length $2k + 1$ such that $\sum_{v\in V(H)}f(v) > k$.
  \end{itemize}
\end{conjecture}

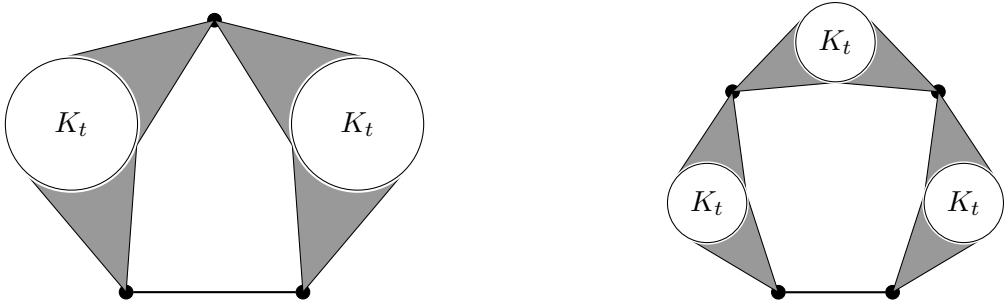
\begin{figure}
  \centering
  \begin{minipage}[t]{.5\textwidth}
    \centering
    \begin{tikzpicture}[scale = 2, remember picture]
      
      \tikzstyle{every node}=[draw, circle];

      \node[scale = .5, fill] (a) at (90:1) {};
      \node[minimum size = 50pt] (b) at (90 + 72:1) {$K_t$};
      \node[scale = .5, fill] (c) at (90 + 72*2:1) {};
      \node[scale = .5, fill] (d) at (90 + 72*3:1) {};
      \node[minimum size = 50pt] (e) at (90 + 72*4:1) {$K_t$};
      \draw[thick] (c) -- (d);

      \tikzstyle{reverseclip}=[insert path={(current page.north east) --
        (current page.south east) --
        (current page.south west) --
        (current page.north west) --
        (current page.north east)}
      ]

      \begin{pgfinterruptboundingbox} 
        \path [clip] (b.center) circle (13pt) [reverseclip];
        \path [clip] (e.center) circle (13pt) [reverseclip];

      \end{pgfinterruptboundingbox}

      \filldraw[fill=white!60!black] (a.center)  -- (tangent cs:node=b,point={(a)},solution=1) -- (tangent cs:node=b,point={(a)},solution=2) -- cycle;
      
      \filldraw[fill=white!60!black] (a.center)  -- (tangent cs:node=e,point={(a)},solution=1) -- (tangent cs:node=e,point={(a)},solution=2) -- cycle;
      \filldraw[fill=white!60!black] (c.center)  -- (tangent cs:node=b,point={(c)},solution=1) -- (tangent cs:node=b,point={(c)},solution=2) -- cycle;
      \filldraw[fill=white!60!black] (d.center)  -- (tangent cs:node=e,point={(d)},solution=1) -- (tangent cs:node=e,point={(d)},solution=2) -- cycle;

    \end{tikzpicture}
  \end{minipage}%
  \begin{minipage}[t]{.5\textwidth}
    \centering
        \begin{tikzpicture}[scale = 1.75, remember picture]
      
          \tikzstyle{every node}=[draw, circle];
          \tikzstyle{vtx} = [scale=.5, fill];
          \tikzstyle{blowup} = [minimum size = 30pt, label=center:$K_t$];
          
          \node[blowup] (a) at (90:1) {};
          \node[vtx] (b) at (90 + 360/7:1) {};
          \node[blowup] (c) at (90 + 2*360/7:1) {};
          \node[vtx] (d) at (90 +3*360/7:1) {};
          \node[vtx] (e) at (90 + 4*360/7:1) {};
          \node[blowup] (f) at (90 + 5*360/7:1) {};
          \node[vtx] (g) at (90 + 6*360/7:1) {};

          \draw[thick] (d) -- (e);

          \tikzstyle{reverseclip}=[insert path={(current page.north east) --
            (current page.south east) --
            (current page.south west) --
            (current page.north west) --
            (current page.north east)}
          ]

          \begin{pgfinterruptboundingbox} 
            \path [clip] (a.center) circle (9pt) [reverseclip];
            \path [clip] (c.center) circle (9pt) [reverseclip];
            \path [clip] (f.center) circle (9pt) [reverseclip];
          \end{pgfinterruptboundingbox}

          \filldraw[fill=white!60!black] (b.center)  -- (tangent cs:node=a,point={(b)},solution=1) -- (tangent cs:node=a,point={(b)},solution=2) -- cycle;
          \filldraw[fill=white!60!black] (g.center)  -- (tangent cs:node=a,point={(g)},solution=1) -- (tangent cs:node=a,point={(g)},solution=2) -- cycle;

          \filldraw[fill=white!60!black] (b.center)  -- (tangent cs:node=c,point={(b)},solution=1) -- (tangent cs:node=c,point={(b)},solution=2) -- cycle;
          \filldraw[fill=white!60!black] (d.center)  -- (tangent cs:node=c,point={(d)},solution=1) -- (tangent cs:node=c,point={(d)},solution=2) -- cycle;

          \filldraw[fill=white!60!black] (e.center)  -- (tangent cs:node=f,point={(e)},solution=1) -- (tangent cs:node=f,point={(e)},solution=2) -- cycle;
          \filldraw[fill=white!60!black] (g.center)  -- (tangent cs:node=f,point={(g)},solution=1) -- (tangent cs:node=f,point={(g)},solution=2) -- cycle;
    \end{tikzpicture}

  \end{minipage}
  
  \caption{Blowups of odd cycles}
  \label{fig:blowups}
\end{figure}

\subsection{Outline of the paper}

Besides the proof of Corollary~\ref{strong turan stability} in Appendix~\ref{strong turan section}, the remainder of the paper is devoted to the proof of Theorem~\ref{main thm}.  In Sections~\ref{prelim-section}-\ref{base clique structure} we prove various properties of a hypothetical minimum counterexample to Theorem~\ref{main thm}.  In Section~\ref{main proof section} we show that a graph cannot simultaneously have these properties, thus demonstrating that a hypothetical counterexample to Theorem~\ref{main thm} does not exist, completing the proof.  A few of the lemmas in Sections~\ref{prelim-section}-\ref{main proof section}, as in the proof of Corollary~\ref{strong turan stability}, utilize standard ``calculus'' arguments to determine a bound on the value of a differentiable multivariate function over a certain region.  In order to not disrupt the flow of the more important ``combinatorial'' arguments in the proof, we collect these ``calculus'' arguments in Appendix~\ref{claim proof section}.

\section{Basic properties of a minimum counterexample}\label{prelim-section}

For the next three sections (that is, Sections~\ref{prelim-section},~\ref{overcoloring section}, and~\ref{base clique structure}), we suppose $G$ is a hypothetical minimum counterexample to Theorem~\ref{main thm} with respect to a function $f$, that is a graph with $\alpha(G) < \sum_{v\in V(G)} f(v)$ where $f : V(G)\rightarrow \mathbb R$ satisfies $f(v) \leq 1/(d(v) + 1/2)$ for each $v\in V(G)$ and $\sum_{v\in K}f(v) \leq 1$ for each clique $K$ in $G$, and subject to these requirements, we assume $G$ is chosen with the minimum number of vertices.
In this section, we prove some basic properties of the minimum counterexample $G$ that will be used frequently in Sections~\ref{overcoloring section} and~\ref{base clique structure}.

For convenience, we let $\delta$ denote the minimum degree of $G$.  We also frequently use the following standard graph theory notation: for a graph $H$ and a vertex $v\in V(H)$, we denote the \textit{open neighborhood} of $v$ in $H$ by $N_H(v)$ and the \textit{closed neighborhood} of $v$ in $H$ by $N_H[v]$.  That is, $N_H(v)$ is the set of vertices in $H$ adjacent to $v$, and $N_H[v] = N_H(v)\cup\{v\}$.  We often omit the subscript $H$ when it is clear from the context.

\begin{lemma}\label{neighborhood demand bound}\label{no simplicial vertex}
  Every vertex $v\in V(G)$ satisfies $f(v) + \sum_{u\in N(v)}f(u) > 1$.  In particular, there are no simplicial vertices in $G$ and $\delta \geq 2$.
\end{lemma}
\begin{proof}
  By the minimality of $G$, we have $\alpha(G - N[v]) \geq \sum_{u\in V(G-N[v])}f(u)$, and since $\alpha(G) \geq \alpha(G - N[v]) + 1$, we have $\alpha(G) \geq 1 + \sum_{u\in V(G-N[v])}f(u)$.  Since $G$ is a counterexample, $\alpha(G) < \sum_{u\in V(G)} f(u)$, so $\sum_{u\in V(G)}f(u) > 1 + \sum_{u\in V(G - N[v])}f(u)$.  By cancelling the terms of the sum $\sum_{u\in V(G - N[v])}f(u)$ in this inequality, we obtain the first inequality.

  Now if $v$ is a simplicial vertex in $G$, then $G[N[v]]$ is a clique, so $f(v) + \sum_{u\in N(v)}f(u) \leq 1$, a contradiction.  Since a vertex of degree one is simplicial, $\delta \geq 2$.
\end{proof}

The argument in Lemma~\ref{neighborhood demand bound} that $f(v) + \sum_{u\in N(v)} f(u) > 1$ actually holds for a minimum counterexample in a more general context, regardless of any stipulations on the function $f$.  Thus, this simple lemma can be used to prove the Caro-Wei Theorem by letting $f(v) = 1/(d(v) + 1)$ for each vertex $v$, as $f(v) + \sum_{u\in N(v)}f(u) \leq 1$ when $v$ has minimum degree.  Indeed, this is essentially Wei's~\cite{W81} original proof (see~\cite[p.~23]{G83}).  Our strategy also focusses on the minimum degree vertices.  The next lemma implies that at least roughly half of the neighbors of a minimum degree vertex also have minimum degree, but we need the following more general form.

\begin{lemma}\label{demand at least one many min degree lemma}\label{not many bigger degree neighbors}
  Suppose $X\subseteq V(G)$ such that $\sum_{v\in X}f(v) > 1$.  If $d(v) \geq |X| - 1$ for every $v\in X$, then fewer than $(|X| + 1/2)/2$ vertices in $X$ have degree at least $|X|$.  In particular, if $d(v) \leq d(u)$ for all $u\in N(v)$, then fewer than $(d(v) + 3/2)/2$ neighbors of $v$ have degree greater than $d(v)$.
\end{lemma}
\begin{proof}
  Let $X'\subseteq X$ be the vertices in $X$ of degree $|X| - 1$ in $G$.  Now
  \begin{equation*}
    1 < \sum_{v\in X} f(v) \leq \frac{|X'|}{|X| - 1/2} + \frac{|X| - |X'|}{|X| + 1/2} = \frac{|X|(|X| - 1/2) + |X'|}{(|X| - 1/2)(|X| + 1/2)}.
  \end{equation*}
  Therefore
  \begin{equation*}
    |X'| > (|X| - 1/2)/2,
  \end{equation*}
  so fewer than
  \begin{equation*}
    |X| - (|X| - 1/2)/2 = (|X| + 1/2)/2
  \end{equation*}
  vertices have degree at least $|X|$, as desired.

  Moreover, if $d(v) \leq d(u)$ for all $u\in N(v)$ and $X = N[v]$, then by assumption, $d(w) \geq |X| - 1$ for every $w\in X$, and by Lemma~\ref{neighborhood demand bound}, $\sum_{v\in X}f(v) > 1$. Thus, fewer than $(d(v) + 3/2)/2$ neighbors of $v$ have degree greater than $d(v)$, as desired.
\end{proof}

\begin{lemma}\label{min degree form clique}
  If $u$ and $w$ are non-adjacent vertices, then $f(u) + f(w) + \sum_{v\in N(u)\cup N(w)}f(v) > 2$.  In particular, if $u$ and $w$ further satisfy $d(u) \leq d(v)$ for all $v\in N(u)$ and $d(w) \leq d(v)$ for all $v\in N(w)$, then $N(u)$ and $N(w)$ are disjoint.
\end{lemma}
\begin{proof}
  By the minimality of $G$, we have $\alpha(G - (N[u]\cup N[w])) \geq \sum_{v\in V(G - (N[u]\cup N[w]))}f(v)$, and since $\alpha(G) \geq \alpha(G - (N[u]\cup N[w])) + 2$, we have $\alpha(G) \geq 2 + \sum_{v\in V(G - (N[u]\cup N[w]))}f(v)$.  Since $G$ is a counterexample, $\alpha(G) < \sum_{v\in V(G)}f(v)$, so $\sum_{v\in V(G)} f(v) > 2 + \sum_{v\in V(G - (N[u]\cup N[w]))}f(v)$.  By cancelling the terms of $\sum_{v\in V(G - (N[u]\cup N[w]))}f(v)$, we obtain $f(u) + f(w) + \sum_{v\in N(u)\cup N(w)}f(v) > 2$, as desired.
    
  Now suppose to the contrary that $d(u) \leq d(v)$ for all $v\in N(u)$, that $d(w) \leq d(v)$ for all $w\in N(w)$, and moreover that $u$ and $w$ share a common neighbor, $x$.  Now we have
  \begin{equation}\label{min degree common nbr equation}
    2 < f(u) + f(w) + \sum_{v\in N(u)\cup N(w)}f(v) \leq f(x) + \sum_{v\in N[u]\setminus\{x\}}f(v) + \sum_{v\in N[w]\setminus\{x\}}f(v) .
  \end{equation}
  Since $d(u) \leq d(v)$ for all $v\in N(u)$, we have $\sum_{v\in N[u]\setminus\{x\}}f(v) \leq d(u)/(d(u) + 1/2)$, and similarly, $\sum_{v\in N[w]\setminus\{x\}}f(v) \leq d(w)/(d(w) + 1/2)$.  We assume without loss of generality that $d(u) \leq d(w)$, in which case $d(u)/ (d(u) + 1/2) \leq d(w)/ (d(w) + 1/2)$.  Now by~\eqref{min degree common nbr equation}, we have $2 < 2d(w) / (d(w) + 1/2) + f(x)$, so $f(x) > 1/(d(w) + 1/2) \geq 1/(d(x) + 1/2)$, a contradiction.
\end{proof}

The proofs of Lemmas~\ref{neighborhood demand bound} and~\ref{min degree form clique} may feel similar.  Indeed, there is a common generalization: if $I$ is an independent set, then $\sum_{v\in I}f(v) + \sum_{u \in N(I)}f(u) > |I|$, where $N(I) = \cup_{v\in I}N(v)$.  However, we only need this result when $|I| \in \{1,2\}$.

Importantly, Lemma~\ref{min degree form clique} implies that the minimum degree vertices can be partitioned into cliques with no edges between them, as follows.  We say a \textit{base clique} of $G$ is a maximum cardinality set of vertices of minimum degree that forms a clique in $G$.  Lemma~\ref{min degree form clique} implies that every vertex of minimum degree is contained in a unique base clique.  Moreover, vertices in different base cliques are not adjacent, and Lemma~\ref{demand at least one many min degree lemma} implies that if $K$ is a base clique, then $|K| > \delta + 1 - (\delta + 3/2)/2 = \delta/2 + 1/4$.  Since $|K|$ is an integer, we thus have the following.
\begin{equation}
  \label{eq:size-of-K}
  \text{If $K$ is a base clique, then $|K| \geq (\delta + 1)/2$.}
\end{equation}

The next two sections focus on these base cliques.  The following definitions are crucial to the proof.

\begin{definition}
  Let $K$ be a base clique of $G$.  
  Now, 
  \begin{itemize}
  \item let $A_K$ be the set of vertices not in $K$ that are complete to $K$,
  \item let $U_K$ be the subset of vertices in $V(G)\setminus(A_K\cup K)$ with a neighbor in $K$,
  \item let $\ell_K = \delta + 1 - |K\cup A_K|$ denote the number of neighbors each vertex in $K$ has in $U_K$, and
  \item let $D_K = \max\{|K\cap N(u)| : u\in U_K\}$ if $U_K \neq \varnothing$ and $0$ otherwise.
  \end{itemize}
\end{definition}

Note that if $K$ is a base clique, then by definition, each vertex in $A_K$ has degree at least $\delta + 1$, and by Lemma~\ref{min degree form clique}, each vertex in $U_K$ has degree at least $\delta + 1$ as well.

\section{Averaging over base cliques}\label{overcoloring section}

Note that there is some ``slack'' in the proof of Lemma~\ref{neighborhood demand bound}.  In particular, if a vertex $w\notin N[v]$ has a neighbor in $N(v)$, then $f(w) < 1/(d_{G-N[v]}(w) + 1/2)$, and so it may be possible to find a larger independent set in $G - N[v]$ by increasing the value of $f(w)$.  We need to be careful, though, not to increase $f(w)$ too much (say to $1/(d_{G - N[v]}(w) + 1/2)$), because $w$ may be simplicial and we need each clique $K$ to satisfy $\sum_{u\in K} f(u) \leq 1$.  It is no coincidence that in the 5-cycle, the extremal example for our theorem, the graph obtained from deleting the closed neighborhood of a minimum degree vertex is complete.

In this section, we exploit this ``slack'' in Lemma~\ref{single lost color lemma}.  One of the main novelties of our proof, with Lemma~\ref{min degree form clique} in hand, is the idea to ``average'' this gain from the slack over a base clique, which we utilize in Lemma~\ref{lost color average}.  This motivates the following definition.

\begin{definition}
  Let $K$ be a base clique of $G$, and let $f_K : V(G - K) \rightarrow\mathbb R$ defined by
  \begin{equation*}
    f_K(v) = \left\{
      \begin{array}{l l}
        \frac{1}{d_{G - K}(v) + 1} & \text{if } v \text{ is simplicial in } G - K,\\
        \frac{1}{d_{G - K}(v) + 1/2} & \text{otherwise}.
      \end{array}\right.
  \end{equation*}
\end{definition}

If $K$ is a base clique of $G$, then since every vertex $v \in V(G - K)$ satisfies $f_K(v) \leq 1/(d_{G - K}(v) + 1/2)$ and every clique $K'$ satisfies $\sum_{v\in K'}f_K(v) \leq 1$, by the minimality of $G$, we have $\alpha(G - K) \geq \sum_{v\in V(G - K)}f_K(v)$.  We use this fact in the next lemma.

\begin{lemma}\label{single lost color lemma}
  If $K$ is a base clique and $v\in K$, then
  \begin{equation}\label{single lost color equation}
  \frac{\ell_K + 1/2}{\delta + 3/2 } - \frac{|K|}{(\delta + 3/2)(\delta + 1/2)} < \sum_{u\in U_K}f(u) - \sum_{u\in U_K\setminus N(v)}f_K(u).
\end{equation}
\end{lemma}
\begin{proof}
  For convenience, let $\ell = \ell_K$, $U = U_K$, and $A = A_K$.
  Since $\alpha(G - N[v]) + 1 \leq \alpha(G) < \sum_{u\in V(G)}f(u)$ and $\alpha(G - N[v]) \geq \sum_{u\in V(G - N[v])}f_K(u)$, we have $\sum_{u\in V(G - N[v])}f_K(u) \leq \sum_{u\in V(G)}f(u) - 1$.  Rearranging terms, we have
  \begin{equation*}
    1 - \sum_{u\in K}f(u) - \sum_{u\in N(v)\cap A}f(u) < \sum_{u\in U}f(u) - \sum_{u\in U\setminus N(v)}f_K(u).
  \end{equation*}
  Since $f(u) \leq 1/(\delta + 3/2)$ for each $u\in A$ and $f(u) \leq 1/(\delta + 1/2)$ for each $u\in K$, the left side of the previous inequality is at least $\frac{\delta + 3/2 - |A|}{\delta + 3/2} - \frac{|K|}{\delta + 1/2}$.  Since $\delta + 1 - |A| = |K| + \ell$, we have for each $v\in K$,
  \begin{equation*}
    \frac{\ell + 1/2}{\delta + 3/2} - \frac{|K|}{(\delta + 3/2)(\delta + 1/2)} < \sum_{u\in U}f(u) - \sum_{u\in U\setminus N(v)}f_K(u),
  \end{equation*}
  as desired.
\end{proof}

Now we show how to ``average'' the gain in Lemma~\ref{single lost color lemma} over the entire base clique.

\begin{lemma}\label{lost color average}
  If $K$ is a base clique of $G$, then
  \begin{equation}\label{lost color average first inequality}
    \ell_K + 1/2 - \frac{|K|}{\delta + 1/2} < \sum_{u\in U_K}\frac{|K\cap N(u)|(\delta + 3/2 -|K|) + |K|/2}{|K|(\delta + 2 - |K\cap N(u)|)}.
  \end{equation}
\end{lemma}

\begin{proof}
  For convenience, let $\ell = \ell_K$, $U = U_K$, and $A = A_K$.
  Since Lemma~\ref{single lost color lemma} holds for each $v\in K$, the left side of~\eqref{single lost color equation} is at most the average of $\sum_{u\in U}f(u) - \sum_{u\in U\setminus N(v)}f_K(u)$ taken over all $v\in K$.  Therefore
  \begin{equation*}
    \frac{\ell + 1/2}{\delta + 3/2} - \frac{|K|}{(\delta + 3/2)(\delta + 1/2)} < \sum_{u\in U}\left( f(u) - \frac{(|K| - |K\cap N(u)|)f_K(u)}{|K|}\right).
  \end{equation*}
  For each $u\in U$,
  \begin{multline*}
    f(u) - \frac{(|K| - |K\cap N(u)|)f_K(u)}{|K|} \leq \frac{1}{d(u) + 1/2} - \frac{|K| - |K\cap N(u)|}{|K|(d(u) + 1 - |K\cap N(u)|)}\\
    = \frac{|K\cap N(u)|(d(u) + 1/2 - |K|) + |K|/2}{|K|(d(u) + 1 - |K\cap N(u)|)(d(u) + 1/2)}.
  \end{multline*}
  Since the above expression is decreasing as a function of $d(u)$ if $d(u)\geq \delta$, the right side of the above inequality is at most
  \begin{equation*}
    \frac{|K\cap N(u)|(\delta + 3/2 - |K|) + |K|/2}{|K|(\delta + 2 - |K\cap N(u)|)(\delta + 3/2)}
  \end{equation*}
  for each $u\in U$.  Hence
  \begin{equation*}
    \ell + 1/2 - \frac{|K|}{\delta + 1/2} < \sum_{u\in U}\frac{|K\cap N(u)|(\delta + 3/2 -|K|) + |K|/2}{|K|(\delta + 2 - |K\cap N(u)|)},
  \end{equation*}
  as desired.
\end{proof}

\section{Structure around base cliques}\label{base clique structure}

In this section, we prove two important properties of base cliques in Lemmas~\ref{ell not 0} and~\ref{ell at most D}, which we use in the next section to show violate~\eqref{lost color average first inequality}.  The proofs of these lemmas are similar in spirit to Lemmas~\ref{neighborhood demand bound} and~\ref{min degree form clique} in that we demonstrate how an independent set in a ``reduction'' of $G$ can be extended to a larger one in $G$, but the reductions are more involved in this section.  In particular, the reductions of $G$ are no longer induced subgraphs: in Lemma~\ref{ell not 0}, we ``identify'' a pair vertices, and in Lemma~\ref{ell at most D} we add edges.

\begin{lemma}\label{ell not 0}
  If $K$ is a base clique, then $\ell_K > 0$.
\end{lemma}
\begin{proof}
  For convenience, let $A = A_K$ and $\ell = \ell_K$.  Suppose to the contrary that $\ell = 0$.  Now $|K\cup A| = \delta + 1$.  By Lemma~\ref{no simplicial vertex}, $A$ is not a clique, so there exists a pair of non-adjacent vertices $u,w\in A$.
  
  Let $G'$ be the graph obtained from $G - (K\cup A\setminus\{u,w\})$ by identifying $u$ and $w$ into a new vertex, say $z$.  Define the function $f' : V(G') \rightarrow \mathbb R$ for $G'$ in the following way.  For each $v\in V(G')\setminus\{z\}$, let $f'(v) = f(v)$, and let $f'(z) = 0.5/(d_G(u) + d_G(w) - 2|K| + 0.5)$.

  Note that for each $v\in V(G')$, we have $f'(v) \leq 1/(d_{G'}(v) + 0.5)$, and moreover, $f'(z) \leq 0.5/(d_{G'}(z) + 0.5)$.
  We claim that for each clique $K'$ in $G'$, we have $\sum_{v\in K'}f'(v) \leq 1$.
  If $z\notin K'$, this inequality holds because $K'$ is a clique in $G$ and $f'(v) = f(v)$ for all $v\in K'$.
  If $z\in K'$, then for each $v\in K'$, we have $d_G(v) \geq |K'| - 1$.  Therefore $f'(v) \leq (|K'| - 0.5)^{-1}$ and $f'(z) \leq 0.5(|K'| - 0.5)^{-1}$, so $\sum_{v\in K'} f'(v) \leq (|K'| - 1)/(|K'| - 0.5) + (1 - 0.5)/(|K'| - 0.5) \leq 1$, as claimed.  Hence, since $G$ is a minimum counterexample, $\alpha(G') \geq \sum_{v\in V(G')}f'(v)$.

  If $I$ is an independent set in $G'$ containing $z$, then $(I\setminus\{z\})\cup \{u, w\}$ is an independent set in $G$ of size $|I| + 1$, and if $I$ is an independent set in $G'$ not containing $z$ and if $v\in K$, then $I\cup\{v\}$ is an independent set in $G$ of size $|I| + 1$.  Therefore $\alpha(G) \geq \alpha(G') + 1$.  Since $G$ is a counterexample, we have $\sum_{v\in V(G)} f(v) > 1 + \sum_{v\in V(G')}f'(v)$, so $\sum_{v\in K\cup A} f(v) - f'(z) > 1$.  Since $f(v) \leq 1/(\delta + 0.5)$ for every $v \in K\cup A$ and $1 - (\delta - 1)/(\delta + 0.5) = 1.5/(\delta + 0.5)$, the previous inequality implies that
  \begin{equation*}
    \frac{1.5}{\delta + 0.5} - f(u) - f(w) + f'(z) < 0.
  \end{equation*}
  This inequality contradicts~\eqref{eq:size-of-K} and the followng claim, which we prove in Appendix~\ref{claim proof section}.
  \begin{claim}\label{sink clique lemma claim}
    Let
    \begin{equation*}
      q_\delta(d_u, d_{w}, k) = \frac{1.5}{\delta + 0.5} - \frac{1}{d_u + 0.5} - \frac{1}{d_{w} + 0.5} + \frac{0.5}{d_u + d_{w} - 2k + 0.5}.
    \end{equation*}
    For $k \geq (\delta + 1)/2$, and $d_u, d_{w} \geq \delta + 1$, we have $q_\delta(d_u, d_{w}, k) \geq 0$.
  \end{claim}
\end{proof}

\begin{lemma}\label{ell at most D}
  If $K$ is a base clique, then $D_K \leq \ell_K \leq (\delta + 1)/2$.  Moreover, $\ell_K \geq 2$ and $\delta \geq 3$.  
\end{lemma}
\begin{proof}
  For convenience, let $U = U_K$, $\ell = \ell_K$, and $D = D_K$.
  First note that by~\eqref{eq:size-of-K}, since $|K| + \ell \leq \delta + 1$, we have $\ell \leq (\delta + 1) / 2$, as desired.

  We actually show $D < \ell$, so suppose for a contradiction that $\ell \leq D$.  We first show that $D = \ell = \delta / 2$.  By Lemma~\ref{ell not 0}, $U \neq \varnothing$, so let $u \in U$ have $D$ neighbors in $K$, let $v \in K\cap N(u)$, and let $v' \in K\setminus N(u)$.  
  
  Let $G'$ be the graph obtained from $G - (N[v]\setminus\{u\})$ by adding an edge between $u$ and each vertex $w\in U\cap N(v')$ if one was not already present.
  Now if $I$ is an independent set in $G'$ containing $u$, then $I \cap N(v') = \varnothing$.  In this case, $I\cup\{v'\}$ is an independent set in $G$ of size $|I| + 1$.  If $I$ is an independent set in $G'$ not containing $u$, then $I \cup \{v\}$ is an independent set in $G$ of size $|I| + 1$.  Thus, $\alpha(G) \geq \alpha(G') + 1$.

  Since $G$ is a counterexample, $\alpha(G')  + 1 < \sum_{v\in V(G)}f(v)$, and since $\sum_{v\in V(G)}f(v) \leq \sum_{v\in V(G')}f(v) + \delta / (\delta + 1/2) \leq \sum_{v\in V(G')}f(v) + 1$, we have $\alpha(G') < \sum_{v\in V(G')} f(v)$.  By the choice of $G$, either there is a vertex $x \in V(G')$ such that $f(x) \leq 1/(d_{G'}(x) + 1/2)$, or there is a clique $K'$ in $G'$ such that $\sum_{x\in K'} f(x) > 1$.  Since $\ell \leq D$, we have $d_{G'}(u) \leq d_G(u)$, and for every vertex $w \in U \cap N(v')$, we also have $d_{G'}(w) \leq d_G(w)$.  Every vertex in $G'$ not $u$ or adjacent to $v'$ has the same neighborhood in $G$ as it does in $G'$.  Therefore every $x \in V(G')$ satisfies $d_{G'}(x) \leq d_G(x)$, so $f(x) \leq 1/(d_{G'}(x) + 1/2)$.  It follows that there is a clique $K'$ in $G'$ such that $\sum_{x\in K'}f(x) > 1$, and this clique contains $u$ and at least one neighbor of $v'$.

  Let $W = \{w \in K' : d_{G'}(w) = |K'| - 1\}$, that is the vertices of $K'$ that are simplicial in $G'$.  By Lemma~\ref{demand at least one many min degree lemma} applied with $X = K'$, we have $|K'\setminus W| < (|K'| + 1/2)/2$, so $|W| > (|K'| - 1/2)/2$.  By Lemma~\ref{min degree form clique}, $d_G(u) \geq \delta + 1$, so $|K'| \geq \delta + 2$.  Therefore $|W| > (|K'| - 1/2)/2 \geq (\delta + 3/2)/2$.  Since $|W|$ is an integer, $|W| \geq \delta / 2 + 1$.  By Lemma~\ref{min degree form clique}, if $w \in W$, then $w$ is adjacent to either $v$ or $v'$, because otherwise $d(w) \leq d(x)$ for all $x \in N(w)$ and $w$ shares $u$ as a neighbor with $v$ in $G$.  Thus, $|W| \leq \ell + 1$, and since $|W| \geq \delta / 2 + 1$, we have $\ell \geq \delta / 2$.

  If $\ell = (\delta + 1)/2$, then $|K| = (\delta + 1)/2$, and since $D < |K|$, we have $D < \ell$, a contradiction.  Therefore $\ell = \delta/2$, as claimed.  Moreover, $|K| \leq \delta/2 + 1$, and since $D \leq |K| - 1$, we have $D \leq \delta / 2$.  Since $D \geq \ell$, we also have $D = \delta/2$, as claimed.

  Now we have shown $|W| \geq \delta / 2 + 1$ and $|W| \leq \ell + 1 = \delta / 2 + 1$, so $|W| = \delta / 2 + 1$.  Moreover, we have $|K'| \geq \delta + 1$ and $|K'| < 2|W| + 1/2 = \delta + 5/2$, and since $|K'|$ is an integer, we have $|K'| \leq \delta + 2$, so $|K'| = \delta + 2$.  Therefore $d(u) = \delta + 1$, and $N(v') \cap U \subseteq W$.  Now $u$ and $v'$ are non-adjacent vertices such that
  \begin{multline*}
    f(u) + f(v') + \sum_{x\in N(u) \cup N(v')}f(x) \leq \sum_{x\in K}f(x) + \sum_{x\in K'}f(x) \leq \frac{\delta/2 + 1}{\delta + 1/2} + \frac{\delta + 2}{\delta + 3/2}\\
    = 2 - \frac{2\delta^2 - \delta - 4}{(2\delta + 1)(2\delta + 3)} < 2,
  \end{multline*}
  contradicting Lemma~\ref{min degree form clique}.  Therefore $D < \ell$, as desired.  Moreover, if $\delta = 2$, then $\ell = D = 1$, a contradiction.  Thus, $\delta \geq 3$, as desired.  By Lemma~\ref{ell not 0}, $D \geq 1$, so $\ell \geq 2$, as desired.
\end{proof}

\section{Proof of Theorem~\ref{main thm}}
\label{main proof section}

In this section we prove Theorem~\ref{main thm}.  The result follows easily by combining Lemmas~\ref{lost color average} and~\ref{ell at most D} with the following lemma.

\begin{lemma}\label{finishing blow}
  Let $G$ be a graph of minimum degree $\delta$ with base clique $K$.  If $D_K\leq \ell_K \leq (\delta + 1)/2$, $\ell_K \geq 2$, and $\delta \geq 3$, then
  \begin{equation*}
    \ell_K + 1/2 - \frac{|K|}{\delta + 1/2} \geq \sum_{u\in U_K}\frac{|K\cap N(u)|(\delta + 3/2 -|K|) + |K|/2}{|K|(\delta + 2 - |K\cap N(u)|)}.
  \end{equation*}
  That is,~\eqref{lost color average first inequality} does not hold.
\end{lemma}

First, we need the following lemma which provides a bound on the right side of the inequality~\eqref{lost color average first inequality}.
\begin{lemma}\label{lost color average implication}
  If $K$ is a base clique, $D'$ is an integer such that $D' \geq D_K$, and $r$ is the smallest positive integer such that $\ell_K|K| \equiv r \mod D'$, then 
  \begin{multline*}
    \sum_{u\in U_K}\frac{|K\cap N(u)|(\delta + 3/2 -|K|) + |K|/2}{|K|(\delta + 2 - |K\cap N(u)|)} \\
    \leq \left\lfloor \frac{\ell|K|}{D'}\right\rfloor\frac{D'(\delta + 3/2 - |K|) + |K|/2}{|K|(\delta + 2 - D')} + \frac{r(\delta + 3/2 - |K|) + |K|/2}{|K|(\delta + 2 - r)}.
  \end{multline*}
\end{lemma}
\begin{proof}
  For convenience, let $\ell = \ell_K$ and $D = D_K$.  
  Let 
  \begin{equation*}
    g(x_1, \dots, x_{\ell|K|}) = \sum_{i=1}^{\ell|K|}\frac{x_i(\delta + 3/2 - |K|) + |K|/2}{|K|(\delta + 2 - x_i)}.
  \end{equation*}
  Note that if $x_2 \geq x_1$, then
  \begin{multline*}
    g(x_1 - 1, x_2 + 1, x_3, \dots, x_{\ell|K|}) - g(x_1, \dots, x_{\ell|K|}) =\\
    \frac{(\delta + 2)(\delta + 3/2 - |K|)}{|K|}\left(\frac{1}{(\delta + 2 - x_2)(\delta + 1 - x_2)} - \frac{1}{(\delta + 3 - x_1)(\delta + 2 - x_1)}\right) > 0,
  \end{multline*}
  Therefore if $1 \leq x_1 \leq x_2 \leq D' - 1$, then
  \begin{equation}\label{switching x_is}
    g(x_1 - 1, x_2 + 1, x_3, \dots, x_{\ell|K|}) > g(x_1, \dots, x_{\ell|K|}).
  \end{equation}
  Now let
  \begin{equation*}
    x_i = \left\{
      \begin{array}{l l}
        D' & i\in\{1, \dots, \lfloor \ell|K|/D'\rfloor\},\\
        r & i = \lfloor \ell|K|/D'\rfloor + 1,\\
        0 & \text{otherwise}.
      \end{array}\right.
  \end{equation*}             
  By~\eqref{switching x_is}, since $\sum_{u\in U}|K\cap N(u)| = \ell |K|$, the left side of the desired inequality is at most $g(x_1, \dots, x_{\ell|K|})$, which is the right side, so the result follows.
\end{proof}

We also need the following claim, proved in Appendix~\ref{claim proof section}.

\begin{claim}\label{ell at most D average color}
  Let
  \begin{equation*}
    q_\delta(\ell, k) = \ell + 0.5 - \frac{k}{\delta + 0.5} - \frac{\ell(\delta + 1.5 - k) + 0.5k}{\delta + 2 - \ell}.
  \end{equation*}
  If $\ell \in [2, \delta/2]$, $k \geq (\delta + 1)/2$, and $\delta \geq 4$, then $q_\delta(\ell, k) \geq 0$.  Also, if $\ell = k = (\delta + 1)/2$ and $\delta \geq 5$ is an odd integer, then
  \begin{equation*}
    \ell + 0.5 - \frac{k}{\delta + 0.5} - \left\lceil\frac{\ell k}{0.5(\delta - 1)}\right\rceil \left(\frac{0.5(\delta - 1)(\delta + 1.5 - k) + 0.5k}{k(\delta + 2 - 0.5(\delta - 1))}\right) > 0.
  \end{equation*}
\end{claim}

Now we can prove Lemma~\ref{finishing blow}.

\begin{proof}[Proof of Lemma~\ref{finishing blow}]
  For convenience, let $\ell = \ell_K$, $D = D_K$, and $U = U_K$.
  Suppose for a contradiction that the inequality does not hold.
  Thus, since $D \leq \ell$, by Lemma~\ref{lost color average implication} (with $D' = \ell$), we have
  \begin{equation*}
    \ell + 1/2 - \frac{|K|}{\delta + 1/2} < \frac{\ell(\delta + 3/2 - |K|) + |K|/2}{\delta + 2 - \ell}.
  \end{equation*}

    Note that the difference of the left and right side of the above inequality is $q_\delta(\ell, |K|)$ from Claim~\ref{ell at most D average color}.  Hence, by Claim~\ref{ell at most D average color}, we may assume either $\delta \leq 3$ or $\ell > \delta/2$.

  First, suppose $\delta \leq 3$.  Hence, $\delta = 3$, $\ell = 2$, $|K| = 2$, and $D = 1$.  Now the right side of~\eqref{lost color average first inequality} is $1.75$ and the left side is $2.5 - 2/3.5 = 27/14$, a contradiction.

  Therefore we may assume $\delta \geq 4$ and $\ell > \delta/2$, and thus, since $\ell$ is an integer at most $(\delta + 1)/2$, we have $\ell = (\delta + 1)/2$.  Hence, $|K| = (\delta + 1)/2$ and $D \leq (\delta - 1)/2$.  Moreover, $\delta$ is odd, and since $\delta \geq 4$, we have $\delta \geq 5$.  Therefore by Lemma~\ref{lost color average implication} applied with $D' = (\delta - 1)/2$, we have
  \begin{equation*}
    \ell + 0.5 - \frac{|K|}{\delta + 0.5} < \left\lceil\frac{\ell |K|}{0.5(\delta - 1)}\right\rceil \left(\frac{0.5(\delta - 1)(\delta + 1.5 - |K|) + 0.5|K|}{|K|(\delta + 2 - 0.5(\delta - 1))}\right), 
  \end{equation*}
  contradicting Claim~\ref{ell at most D average color}.
\end{proof}

As mentioned, the proof of Theorem~\ref{main thm} follows easily from Lemmas~\ref{lost color average},~\ref{ell at most D}, and~\ref{finishing blow}.

\appendix
\section{Strong Tur\'{a}n stability}\label{strong turan section}

\begin{proof}[Proof of Corollary~\ref{strong turan stability}]
  Suppose for the sake of contradiction that $G$ is a $K_{r + 1}$-free graph on $n$ vertices with chromatic number at least $r + 1$ and at least $(1 - 1/r)n^2/2 - n/(2r) + 1$ edges.  Note that the case $r \leq 1$ is vacuous, so we assume $r \geq 2$.

  Since $r \geq 2$, by Theorem~\ref{strong local turan stability} with $\sigma = 1/2$, $G$ contains an independent set $I$ with a vertex complete to $G - I$.  Let $I_1, \dots, I_k$ be a set of pairwise disjoint independent sets in $G$ chosen with $k$ maximal subject to for each $i\in\{1, \dots, k\}$, the independent set $I_i$ contains a vertex $v_i$ complete to $G - \cup_{j=1}^i I_j$, let $G' = G - \cup_{i=1}^k I_i$, and let $n'$ be the number of vertices of $G'$.  If $K$ is a clique in $G'$, then $K\cup\{v_1, \dots, v_k\}$ is a clique in $G$, so $G'$ is $K_{r + 1 - k}$-free.
  Thus, if $k = r - 1$, then $V(G')$ is an independent set and $G$ is $r$-colorable, a contradiction.  Therefore we assume $k \leq r - 2$.  Since $k\geq 1$, we have $r \geq 3$.

  Therefore by Theorem~\ref{strong local turan stability} and the maximality of $k$, we have $|E(G')| \leq (1 - 1/(r - k))n'^2/2 - n'/4$, and by Tur\'{a}n's Theorem, $|E(G[\cup_{i=1}^k I_i])| \leq (1 - 1/k)(n - n')^2/2$.  Thus, since the number of edges between $\cup_{i=1}^k I_i$ and $V(G')$ is at most $(n - n')n'$, we have 
  \begin{equation}\label{stability proof edge bound}
    |E(G)| \leq \left(1 - \frac{1}{k}\right)\frac{(n - n')^2}{2} + \left(1 - \frac{1}{r - k}\right)\frac{n'^2}{2} - \frac{n'}{4} + (n - n')n'.
  \end{equation}
  
  The remainder of the proof amounts to solving a simple optimization problem -- we derive a contradiction by showing that the right side of~\eqref{stability proof edge bound} is at most $(1 - 1/r)n^2/2 - n/(2r) + 1$ for $n' \leq n$ and $k \leq r - 2$.  To that end, let $e_{n,r}(n', k)$ denote the right side of~\eqref{stability proof edge bound} and treat $n'$ and $k$ as real-valued variables.
  Note that
  \begin{equation*}
    \frac{\partial}{\partial n'}e_{n, r}(n', k) = -\left(1 - \frac{1}{k}\right)(n - n') + n - 2n' + \left(1 - \frac{1}{r - k}\right)n' - \frac{1}{4} = \frac{n}{k} - \frac{1}{4} - n'\left(\frac{1}{k} + \frac{1}{r-k}\right),
  \end{equation*}
  and thus $e_{n, r}(n', k)$ is maximized when
  \begin{equation*}
    n' = \frac{r - k}{r}\left(n - \frac{k}{4}\right).
  \end{equation*}
  Hence, it suffices to show that $e_{n, r}((r - k)(n - k/4)/r, k) \leq (1 - 1/r)n^2/2 - n/(2r) + 1$ for $k \leq r - 2$, and this inequality holds for $r > 0$ so long as
  \begin{equation*}
    k^2 - 8kn - kr + 8nr - 16n + 32r > 0.
  \end{equation*}
  The left side of the above inequality is minimized when $k = 4n + r/2$, and since $k \leq r - 2 \leq 4n + r/2$, we only need that the above inequality holds for $k = r - 2$.  Indeed,
  \begin{equation*}
    (r - 2)^2 - 8(r - 2)n - (r - 2)r + 8nr - 16n + 32r = 30r + 4 > 0,
  \end{equation*}
  as required.  
\end{proof}

\section{Proving the claims}\label{claim proof section}

In this section we prove the various claims that appear throughout the paper.

\begin{proof}[Proof of Claim~\ref{sink clique lemma claim}]
  First observe that $q_\delta(d_u, d_w, k)$ is increasing in $k$, so it suffices to show that $q'_\delta(d_u, d_w) = q_\delta(d_u, d_w, (\delta + 1)/2) \geq 0$.  Note that
  \begin{equation*}
    \frac{\partial}{\partial d_u}q'_\delta(d_u, d_w) = \frac{1}{(d_u + 0.5)^2} - \frac{0.5}{(d_u + d_w - (\delta + 1) + 0.5)^2}.
  \end{equation*}
  Since $d_w \geq \delta + 1$, the right side of the equality above is at least $1/(d_u + 0.5)^2 - 0.5/(d_u + 0.5)^2 > 0$.  Therefore $q'_\delta(d_u, d_w)$ is increasing as a function of $d_u$, and likewise for $d_w$ by symmetry.  Since $d_u, d_w \geq \delta + 1$, we have $q'_\delta(d_u, d_w) \geq q'_\delta(\delta + 1, \delta + 1)$, and
  \begin{equation*}
    q'(\delta + 1, \delta + 1) = \frac{1.5}{\delta + 0.5} - \frac{1}{\delta + 1.5} - \frac{1}{\delta + 1.5} + \frac{0.5}{\delta + 1.5} = \frac{1.5}{\delta + 0.5} - \frac{1.5}{\delta + 1.5} > 0,
  \end{equation*}
  as desired.
\end{proof}

\begin{proof}[Proof of Claim~\ref{ell at most D average color}]
  First we show that $q_\delta(\ell, k) \geq 0$.
  Note that
  \begin{equation*}
    \frac{\partial}{\partial k}q_\delta(\ell, k) = \frac{-1}{\delta + 0.5} + \frac{\ell - 0.5}{\delta + 2 - \ell} \geq \frac{-1}{\delta + 0.5} + \frac{1.5}{\delta} > 0.
  \end{equation*}
  Therefore $q_\delta(\ell, k) \geq q_\delta(\ell, (\delta + 1)/2)$.  For convenience, let $q'_\delta(\ell) = q_\delta(\ell, (\delta + 1)/2))$, and note that
  \begin{equation*}
    \frac{\partial}{\partial \ell}q'_\delta(\ell) = 1 - \frac{(0.5\delta + 1)(\delta + 2) + .25(\delta + 1)}{(\delta + 2 - \ell)^2}.
  \end{equation*}
  Hence, $\frac{\partial}{\partial \ell}q'_\delta(\ell) \geq 0$ if and only if
  \begin{equation*}
    \ell^2 - 2(\delta + 2)\ell + (\delta + 2)^2 - ((0.5\delta + 1)(\delta + 2) + .25(\delta + 1)) \geq 0.
  \end{equation*}
  By the quadratic formula applied to $\ell$, $\frac{\partial}{\partial \ell}q'_\delta(\ell) \leq 0$ if and only if
  \begin{multline*}
    \frac{1}{2}\left(2(\delta + 2) - \sqrt{(2(\delta + 2))^2 - 4((\delta + 2)^2 - ((0.5\delta + 1)(\delta + 2) + .25(\delta + 1)))}\right) \leq \ell \\
    \leq \frac{1}{2}\left(2(\delta + 2) + \sqrt{(2(\delta + 2))^2 - 4((\delta + 2)^2 - ((0.5\delta + 1)(\delta + 2) + .25(\delta + 1)))}\right).
  \end{multline*}
  Note that
  \begin{equation*}
    (2(\delta + 2))^2 - 4((\delta + 2)^2 - ((0.5\delta + 1)(\delta + 2) + .25(\delta + 1))) = 2\delta^2 + 9\delta + 9 \geq 0.
  \end{equation*}
  Therefore $\frac{\partial}{\partial \ell}q'_\delta(\ell) \leq 0$ if
  \begin{equation*}
    \delta + 2 - 0.5\sqrt{2\delta^2 + 9\delta + 9} \leq \ell \leq \delta + 2
  \end{equation*}
  and $\frac{\partial}{\partial \ell}q'_\delta(\ell) \geq 0$ if $\ell \leq \delta + 2 - 0.5\sqrt{2\delta^2 + 9\delta + 0}$.  
  Therefore $q'_\delta(\ell) \geq \min\{q'_\delta(2), q'_\delta(\delta/2)\}$.  Note that
  \begin{equation*}    
    q'_\delta(2) = -\frac{1.25   \delta + 2.25}{\delta} + \frac{-0.5   \delta - 0.5}{\delta + 0.5} + 2.5 \approx \frac{(\delta - 3.28935) (\delta + 0.456017)}{\delta (\delta + 0.5)} >0,
  \end{equation*}
  and
  \begin{multline*}
    q'_\delta(\delta/2) = 0.5   \delta - \frac{\delta {\left(0.25   \delta + 0.5\right)} + 0.25  \delta + 0.25}{0.5   \delta + 2} + \frac{-0.5   \delta - 0.5}{\delta + 0.5} + 0.5 \\
    \approx \frac{0.5 (\delta - 2.15831) (\delta + 1.15831)}{(\delta + 0.5) (\delta + 4)}>0,
  \end{multline*}
  as desired.

  Now we show the second inequality.
  Since $\delta \geq 5$, we have $4/(2(\delta - 1)) \leq 1/2$, so
  \begin{equation*}
    \left\lceil\frac{\ell k}{0.5(\delta - 1)}\right\rceil = \left\lceil \frac{(\delta + 1)^2}{2(\delta - 1)}\right\rceil  = \left\lceil \frac{(\delta + 3)(\delta - 1) + 4}{2(\delta - 1)}\right\rceil \leq \left\lceil 0.5(\delta + 3) + 0.5 \right\rceil = 0.5(\delta + 5).
  \end{equation*}
  Therefore
  \begin{multline*}
    \ell + 0.5 - \frac{k}{\delta + 0.5} - \left\lceil\frac{\ell k}{0.5(\delta - 1)}\right\rceil\left(\frac{0.5(\delta - 1)(\delta + 1.5 - k) + 0.5k}{k(\delta + 2 - 0.5(\delta - 1)}\right) \geq\\
    0.5(\delta + 1) + 0.5 - \frac{0.5(\delta + 1)}{\delta + 0.5} - 0.5(\delta + 5)\left(\frac{0.5(\delta - 1)(\delta + 1.5 - 0.5(\delta + 1)) + .25(\delta + 1)}{0.5(\delta + 1)(\delta + 2 - 0.5(\delta - 1))}\right)\\
    = \frac{0.75 (\delta + 1/3)}{(\delta + 0.5) (\delta + 1)} > 0,
  \end{multline*}
  as desired.
\end{proof}


\begin{thebibliography}{10}

\bibitem{AKS80}
M.~Ajtai, J.~Koml\'{o}s, and E.~Szemer\'{e}di.
\newblock A note on {R}amsey numbers.
\newblock {\em J. Combin. Theory Ser. A}, 29(3):354--360, 1980.

\bibitem{BGHR12}
P.~Borowiecki, F.~G\"{o}ring, J.~Harant, and D.~Rautenbach.
\newblock The potential of greed for independence.
\newblock {\em J. Graph Theory}, 71(3):245--259, 2012.

\bibitem{BR15}
P.~Borowiecki and D.~Rautenbach.
\newblock New potential functions for greedy independence and coloring.
\newblock {\em Discrete Appl. Math.}, 182:61--72, 2015.

\bibitem{BJ08}
N.~Bougard and G.~Joret.
\newblock Tur{\'a}n's theorem and k-connected graphs.
\newblock {\em J. Graph Theory}, 58(1):1--13, 2008.

\bibitem{CRRS16}
C.~Brause, B.~Randerath, D.~Rautenbach, and I.~Schiermeyer.
\newblock A lower bound on the independence number of a graph in terms of
  degrees and local clique sizes.
\newblock {\em Discrete Appl. Math.}, 209:59--67, 2016.

\bibitem{B41}
R.~L. Brooks.
\newblock On colouring the nodes of a network.
\newblock {\em Proc. Cambridge Philos. Soc.}, 37:194--197, 1941.

\bibitem{Br81}
A.~E. Brouwer.
\newblock {\em Some lotto numbers from an extension of {T}ur\'{a}n's theorem},
  volume 152 of {\em Afdeling Zuivere Wiskunde [Department of Pure
  Mathematics]}.
\newblock Mathematisch Centrum, Amsterdam, 1981.

\bibitem{C79}
Y.~Caro.
\newblock New results on the independence number.
\newblock Technical report, Tel Aviv University, 1979.

\bibitem{E66}
P.~{Erd\H os}.
\newblock On some new inequalities concerning extremal properties of graphs.
\newblock In {\em Theory of Graphs (Proc. Colloq., Tihany, 1966)}, pages
  77--81, 1966.

\bibitem{E67}
P.~{Erd\H os}.
\newblock Some recent results on extremal problems in graph theory.
\newblock {\em Theory of Graphs (Internat. Sympos., Rome, 1966)}, pages
  117--123, 1967.

\bibitem{F15}
Z.~F{\"u}redi.
\newblock {A proof of the stability of extremal graphs, Simonovits' stability
  from Szemer{\'e}di's regularity}.
\newblock {\em J. Combin. Theory Ser. B}, 115:66--71, 2015.

\bibitem{GV10}
I.~Gitler and C.~E. Valencia.
\newblock On bounds for some graph invariants.
\newblock {\em Bol. Soc. Mat. Mexicana (3)}, 16:73--94, 2010.

\bibitem{G83}
J.~R. Griggs.
\newblock Lower bounds on the independence number in terms of the degrees.
\newblock {\em J. Combin. Theory Ser. B}, 34(1):22--39, 1983.

\bibitem{HM19}
J.~Harant and S.~Mohr.
\newblock On {Selkow's} bound on the independence number of graphs.
\newblock {\em Discuss. Math. Graph Theory}, 39(3):655--657, 2019.

\bibitem{HR11}
J.~Harant and D.~Rautenbach.
\newblock Independence in connected graphs.
\newblock {\em Discrete Appl. Math.}, 159(1):79--86, 2011.

\bibitem{HS01}
J.~Harant and I.~Schiermeyer.
\newblock On the independence number of a graph in terms of order and size.
\newblock {\em Discrete Math.}, 232(1-3):131--138, 2001.

\bibitem{HS06}
J.~Harant and I.~Schiermeyer.
\newblock A lower bound on the independence number of a graph in terms of
  degrees.
\newblock {\em Discuss. Math. Graph Theory}, 26(3):431--437, 2006.

\bibitem{KP19-brooks}
T.~Kelly and L.~Postle.
\newblock The local fractional {Brooks'} theorem.
\newblock submitted.

\bibitem{KP18}
T.~Kelly and L.~Postle.
\newblock Fractional coloring with local demands.
\newblock {\em arXiv:1811.11806}, 2018.

\bibitem{N11}
V.~Nikiforov.
\newblock Some new results in extremal graph theory.
\newblock In {\em Surveys in combinatorics 2011}, volume 392 of {\em London
  Math. Soc. Lecture Note Ser.}, pages 141--181. Cambridge Univ. Press,
  Cambridge, 2011.

\bibitem{NR04}
V.~Nikiforov and C.~Rousseau.
\newblock Large generalized books are p-good.
\newblock {\em J. Combin. Theory Ser. B}, 92(1):85--97, 2004.

\bibitem{PSS18}
K.~Popielarz, J.~Sahasrabudhe, and R.~Snyder.
\newblock A stability theorem for maximal {$K_{r+1}$}-free graphs.
\newblock {\em J. Combin. Theory Ser. B}, 132:236--257, 2018.

\bibitem{S94}
S.~M. Selkow.
\newblock A probabilistic lower bound on the independence number of graphs.
\newblock {\em Discrete Math.}, 132(1-3):363--365, 1994.

\bibitem{Sh83}
J.~B. Shearer.
\newblock A note on the independence number of triangle-free graphs.
\newblock {\em Discrete Math.}, 46(1):83--87, 1983.

\bibitem{Sh91}
J.~B. Shearer.
\newblock A note on the independence number of triangle-free graphs. {II}.
\newblock {\em J. Combin. Theory Ser. B}, 53(2):300--307, 1991.

\bibitem{S66}
M.~Simonovits.
\newblock A method for solving extremal problems in graph theory, stability
  problems.
\newblock In {\em Theory of Graphs (Proc. Colloq., Tihany, 1966)}, pages
  279--319, 1966.

\bibitem{T41}
P.~Tur\'{a}n.
\newblock Eine {E}xtremalaufgabe aus der {G}raphentheorie.
\newblock {\em Mat. Fiz. Lapok}, 48:436--452, 1941.

\bibitem{TU15}
M.~Tyomkyn and A.~J. Uzzell.
\newblock Strong {T}ur{\'{a}}n stability.
\newblock {\em Electron. J. Comb.}, 22(3):P3.9, 2015.

\bibitem{W81}
V.~K. Wei.
\newblock A lower bound on the stability number of a simple graph.
\newblock Technical Memorandum 81-11217-9, Bell Laboratories, 1981.

\end{thebibliography}
\end{document}